\documentclass[english]{amsart}

\usepackage{amsmath}
\usepackage{amssymb}
\usepackage{amsthm}
\usepackage{amscd}
\usepackage{babel}
\usepackage[latin1]{inputenc}
\usepackage{graphicx}

\newcommand{\Z}[1]{\mathbb{Z}/#1\mathbb{Z}}

\newtheorem{teor}{Theorem}
\newtheorem{Theorem}{Theorem}

\newtheorem{cor}{Corollary}
\newtheorem{prop}{Proposition}

\newtheorem{Lemma}{Lemma}

\theoremstyle{definition}
\newtheorem{defi}{Definition}
\newtheorem{rem}{Remark}

\renewcommand{\subjclassname}{AMS \textup{2010} Mathematics Subject
Classification\ }

\author{Jos\'{e} Mar\'{i}a Grau}
\address{Departamento de Matemáticas, Universidad de Oviedo\\ Avda. Calvo Sotelo s/n, 33007 Oviedo, Spain}
\email{grau@uniovi.es}

\author{Celino Miguel}
\address{Instituto de Telecomunicaoes, Polo de Covilha}
\email{celino@ubi.pt}

\author{Antonio M. Oller-Marc\'{e}n}
\address{Centro Universitario de la Defensa de Zaragoza\\ Ctra. Huesca s/n, 50090 Zaragoza, Spain}
\email{oller@unizar.es}

\title{Generalized Quaternion  Rings over $\mathbb{Z}/n\mathbb{Z}$ for an odd $n$}

\begin{document}

\begin{abstract}  
We consider a generalization of the quaternion ring $\Big(\frac{a,b}{R}\Big)$ over a commutative unital ring $R$ that includes the case when $a$ and $b$ are not units of $R$. In this paper, we focus on the case $R=\mathbb{Z}/n\mathbb{Z}$ for and odd $n$. In particular, for every odd integer $n$ we compute the number of non-isomorphic generalized quaternion rings $\Big(\frac{a,b}{\mathbb{Z}/n\mathbb{Z}}\Big)$.
\end{abstract}

\maketitle
\subjclassname{11R52,16-99}


\section{Introduction}

The origin of quaternions dates back to 1843, when William Rowan Hamilton considered a $4$-dimensional vector space over $\mathbb{R}$ with basis $\{1,i,j,k\}$ and defined an associative product given by the now classical rules $i^2=j^2=-1$ and $ij=-ji=k$. These so-called ``Hamilton quaternions'' turned out to be the only division algebra over $\mathbb{R}$ with dimension greater than 2.
Later on, this idea was extended to define quaternion algebras over arbitrary fields. Thus, if $F$ is a field and $a,b\in F\setminus\{0\}$ we can define a unital, associative, $4$-dimensional algebra over $F$ just considering a basis $\{1,i,j,k\}$ and the product given by $i^2=a$, $j^2 = b$ and $ij = -ji=k$. The structure of quaternion algebras over fields of characteristic different from two is well-known. Indeed, such a quaternion algebra is either a division ring or isomorphic to the matrix ring $\mathbb{M}_2(F)$ \cite[p.19]{pr}. This is no longer true if $F$ is of characteristic 2, since quaternions over $\Z{2}$ are not a division ring but they form a commutative ring, while $\mathbb{M}_2(\Z{2})$ is not commutative. Nevertheless, some authors consider a different product in the characteristic 2 case given by $i^2 + i = a$, $j^2 = b$, and $ ji = (i + 1)j=k$. The algebra defined by this product is isomorphic to the corresponding matrix ring.

Generalizations of the notion of quaternion algebra to other commutative base rings $R$ have been considered by Kanzaki \cite{ka}, Hahn \cite{ha}, Knus \cite{max}, Gross and Lucianovic \cite{GLU},  Tuganbaev \cite{tu}, and most recently by John Voight \cite{jv,jv3}. Quaternions over finite rings have attracted significant attention since they have applications in coding theory see,  \cite{oz2,oz1,codes}. In \cite{nosotros} the case $R=\mathbb{Z}/n\mathbb{Z}$ was studied proving the following result.

\begin{Theorem}[\cite{nosotros}, Theorem 4]
Let $n$ be an integer and let $a,b$ be such that $\gcd(a,n)=\gcd(b,n)=1$. The following hold:
\begin{itemize}
\item[i)] If $n$ is odd, then
$$\Big(\frac{a,b}{\Z{n}} \Big) \cong \mathbb{M}_2(\Z{n}).$$
\item[ii)] If $n=2^sm$ with $s>0$ and $m$ odd, then
$$\Big(\frac{a,b}{\Z{n}}\Big) \cong \begin{cases}\mathbb{M}_2(\Z{m}) \times (\frac{-1,-1}{\Z{2^s}}), & \textrm{if $s=1$ or $a \equiv b\equiv -1 \pmod{4}$}; \\ \mathbb{M}_2(\Z{m}) \times (\frac{1,1}{\Z{2^s}}), & \textrm{otherwise}.\end{cases}$$
\end{itemize}
\end{Theorem}

In this paper, we extend the concept of quaternion rings over commutative, associative, unital rings to the case when $i^2$ and $j^2$ are not necessarily units of the ring $R$. In particular, we will focus on the case $R=\mathbb{Z}/n\mathbb{Z}$ for an odd $n$.

\section{Basic concepts}

Let $R$ be a commutative and associative ring with identity and let $H(R)$ denote the free $R$-module of rank $4$ with basis $\{1, i, j, k\}$. That is,
\begin{equation*}H(R)=\{x_0+x_1i+x_2j+x_3k\;:\;x_0, x_1, x_2, x_3\in R\}.\end{equation*}
Now, let $a,b\in R$ and define an associative multiplication in $H(R)$ according to the following rules:
\begin{align*}i^2&=a,\\ j^2&=b,\\ ij&=-ji=k.\end{align*}
Thus, we obtain an associative, unital ring called a quaternion ring over $R$ which is denoted by $\Big(\frac{a,b}{R}\Big)$.

\begin{defi}
A \textit{standard basis} of $\Big(\frac{a,b}{R}\Big)$ is any base $\mathcal{B}=\{1,I,J,K\}$ of the free $R$-module $H(R)$ such that \begin{align*}I^2&=a,\\ J^2&=b,\\ IJ&=-JI=K.\end{align*}
Given the standard basis $\{1,i,j,k\}$, the elements of the submodule $R\langle i,j,k\rangle$ are called pure quaternions. Note that the square of a pure quaternion always lays on $R$.
\end{defi}

\begin{rem}
Given $q\in \Big(\frac{a,b}{R}\Big)$ and a fixed standard basis, there exist $x_0\in R$ and a pure quaternion $q_0$ such that $q=x_0+q_0$. Observe that both $x_0$ and $q_0$ are uniquely determined and also that the only pure quaternion in $R$ is 0.
\end{rem}

The following classical concepts are not altered by the fact that $a$ and $b$ are not necessarily units.

\begin{defi}\label{trn}
Consider the standard basis $\{1,i,j,k\}$ and let $q\in\Big(\frac{a,b}{R}\Big)$. Put $q=x_0+q_0$ with $x_0\in R$ and $q_0=x_1i+x_2j+x_3k$ a pure quaternion. Then,
\begin{itemize}
\item[i)] The conjugate of $q$ is: $\overline{q}=x_0-q_0=x_0+x_1i-x_2j-x_3k$.
\item[ii)] The trace of $q$ is $\textrm{tr}(q)=q+\overline{q}=2x_0.$
\item[iii)] The norm of $q$ is $\textrm{n}(q)=q\overline{q}=x_0^2-q_0^2=x_0^2-ax_1^2-bx_2^2+abx_3^2$.
\end{itemize}
Note that $\textrm{n}(q),\textrm{tr}(q)\in R$ and $\textrm{n}(q_1q_2)= \textrm{n}(q_1)\textrm{n}(q_2)$.
\end{defi}

\begin{rem}\label{rempur}
Observe that, if $q$ is a pure quaternion, then $\overline{q}=-q$ and $\textrm{tr}(q)=0$. The converse also holds only if $R$ has odd characteristic. 
\end{rem}

In the following result we will see that isomorphisms preserve conjugation. The classical proof in the case when $a$ and $b$ are units (see \cite[Theorem 5.6]{CON} for instance) is no longer valid in our setting and it must be slightly modified.

\begin{Theorem}\label{teoriso}
Let $f:\Big(\frac{a,b}{R} \Big) \to \Big(\frac{c,d}{R} \Big)$ be a ring isomorphism. Then, for every $q \in \Big(\frac{a,b}{R} \Big)$ it holds that 
$f(\overline{q})=\overline{f(q)}$.
\end{Theorem}
\begin{proof}
Let $q\in\Big(\frac{a,b}{R}\Big)$ and put $q=x_0+q_0$ with $x_0\in R$ and $q_0$ a pure quaternion. Then, $\overline{q}=x_0-q_0$ and $f(\overline{q})=f(x_0)-f(q_0)=x_0-f(q_0)$. On the other hand, $\overline{f(q)}=\overline{f(x_0+q_0)}=\overline{f(x_0)+f(q_0)}=\overline{x_0+f(q_0)}=x_0+\overline{f(q_0)}$. Hence, in order to prove the result, it is enough to prove that $\overline{f(q_0)}=-f(q_0)$ for every pure quaternion $q_0$.

Let us consider the standard basis $\{1,i,j,k\}$ of $\Big(\frac{a,b}{R} \Big)$. Then, $f(i)=\alpha_1+q_1$ with $\alpha_1\in R$ and $q_1$ a pure quaternion in $\Big(\frac{c,d}{R} \Big)$. Now, since $i^2\in R$ and taking into account that $f$ fixes $R$, we have that $f(i^2)=f(i)^2=(\alpha_1+q_1)^2=\alpha_1^2+q_1^2+2\alpha_1q_1\in R$. Consequently, $2\alpha_1q_1\in R$ (because both $\alpha_1^2$ and $q_1^2$ are in $R$) and since $2\alpha_1q_1$ is a pure quaternion, it must be $2\alpha_1q_1=0$. Thus, $f(2\alpha_1i)=2\alpha_1f(i)=2\alpha_i^2$ and, since $f$ fixes $R$, it follows that $2\alpha_1i=0$ and also that $2\alpha_1=0$. Equivalently, $\alpha_1=-\alpha_1$ and then, $\overline{f(i)}=\alpha_1-q_1=-\alpha_1-q_1=-f(i)$.

In the same way, it can be seen that $\overline{f(j)}=-f(j)$ and $\overline{f(k)}=-f(k)$. Thus, if $q_0=Ai+Bj+Ck$ is a pure quaternion in $\Big(\frac{a,b}{R}\Big)$ we have that:
$$
\overline{f(q_0)}=A\overline{f(i)}+B\overline{f(j)}+C\overline{f(k)}=-Af(i)-Bf(j)-Cf(k)=-f(q_0),
$$
and the result follows.
\end{proof}

Since both the trace and the norm are defined in terms of the conjugation, the following result easily follows from Theorem \ref{teoriso}.

\begin{cor}\label{coriso}
Let $f:\Big(\frac{a,b}{R} \Big) \to \Big(\frac{c,d}{R} \Big)$ be a ring isomorphism. Then, for every $q \in \Big(\frac{a,b}{R} \Big)$ the following hold.
\begin{itemize}
\item[i)] ${\rm tr}\Big( f(q)\Big)={\rm tr}(q)$.
\item[ii)] ${\rm n}\Big( f(q)\Big)={\rm n}(q)$.
\end{itemize}
\end{cor}

\begin{rem}
Theorem \ref{teoriso} and Corollary \ref{coriso} imply in particular that the conjugate, the trace and the norm of an element are independent from the standard basis of $\Big(\frac{a,b}{R} \Big)$ used to compute them. Moreover, according to Remark \ref{rempur}, Theorem \ref{teoriso} implies that (in the odd characteristic case) every isomorphism preserves pure quaternions.
\end{rem}

\begin{prop}\label{formamat}
Let $R$ be a ring with odd characteristic and Let $f:\Big(\frac{a,b}{R} \Big)\to  \Big(\frac{a,c}{R} \Big)$ be a ring isomorphism. Then, for some pair of standard bases the matrix of $f$ has the form
$$\begin{pmatrix} 1&0&0&0\\0&1&\alpha_1&\alpha_2\\0&0&\beta_1&\beta_2\\0&0&\gamma_1&\gamma_2\end{pmatrix},$$
with $\alpha_1a=\alpha_2a=0$. 
\end{prop}
\begin{proof}
Let $\{1,i,j,k\}$ be any standard basis in $\Big(\frac{a,b}{R} \Big)$. Since $f(i)^2=f(i^2)=a$, let us consider $S$ the subring of $\Big(\frac{a,c}{R} \Big)$ generated by $\{1,f(i)\}$. Now, we apply the Cayley-Dickson process \cite{SCH} to $S$ and $c$ so that we extend $\{1,f(i)\}$ to an standard basis $\{1,I:=f(i),J,K\}$ of $\Big(\frac{a,c}{R} \Big)$. 

Now, since $R$ has odd characteristic, $f$ preserves pure quaternions. Thus, $f(j)=\alpha_1I+\beta_1J+\gamma_1K$ and $f(k)=\alpha_2I+\beta_2J+\gamma_2K$.

Finally, $f(k)=f(ij)=f(i)f(j)=I(\alpha_1I+\beta_1J+\gamma_1K)=\alpha_1a+\beta_1K+\gamma_1aJ$ must be a pure quaternion and hence $\alpha_1a=0$. In the same way it can be seen that $\alpha_2a=$ and the result follows.
\end{proof}

In what follows, we will be interested in determining whether two different quaternion rings are isomorphic or not. The following isomorphism, which is well-known if $a$ and $b$ are units, also holds in our setting. The proof is straightforward.

\begin{Lemma} \label{orden} 
Let $a,b \in R$. Then, 
$$\Big(\frac{ a, b}{R}\Big) \cong \Big(\frac{ b, a}{R}\Big).$$
\end{Lemma}
 
Nevertheless, some other easy isomorphisms that hold in the case when $a$ and $b$ are units, like 
\begin{equation}\label{noiso} \Big(\frac{ a, b  }{R}\Big) \cong \Big(\frac{ a, -a b  }{R}\Big) \cong \Big(\frac{ b, -a b  }{R}\Big) \end{equation}
are, as we will see, no longer generally true in our setting.

\section{Some results regarding $\Big(\frac{a,b}{\mathbb{Z}/p^k\mathbb{Z}}\Big)$ for a prime $p$}

Throughout this section $p$ will denote any prime. The next two results present some isomorphisms that will be useful in forthcoming sections. The first one (Lemma \ref{menos}) is, in some sense, an analogue to the classical isomorphism (\ref{noiso}). The second one (Lemma \ref{modulo}) presents some kind of descent principle.

\begin{Lemma}\label{menos} 
Let $a$ and $b$ be integers with $\gcd(a,p)=1$. Then,
$$\Big(\frac{ a, b p^s}{\mathbb{Z}/p^k\mathbb{Z}}\Big) \cong \Big(\frac{ a, -a b p^s}{\mathbb{Z}/p^k\mathbb{Z}}\Big).$$
\end{Lemma}
\begin{proof}
Let us consider standard bases $\{1,i,j,k\}$ and $\{1,I,J,K\}$ of $\Big(\frac{ a, b p^s}{\mathbb{Z}_{p^k}}\Big)$ and $\Big(\frac{ a, -a b p^s}{\mathbb{Z}_{p^k}}\Big)$, respectively. Then it is enough to consider the ring homomorphism $f$ induced by $f(1)=1$, $f(I)=i$, $f(J)=k$ and $f(K)=aj$; which is bijective because its coordinate matrix
$$\begin{pmatrix} 1&0&0&0\\0&1&0&0\\0&0&0&a\\0&0&1&0\end{pmatrix}$$
is regular over $\mathbb{Z}/p^k\mathbb{Z}$.
\end{proof}

\begin{Lemma} \label{modulo}
Let $a_i$ ($1\leq i\leq 4$) and $k\geq 1$ be integers such that $$ \Big(\frac{a_1,a_2}{\mathbb{Z}/p^k\mathbb{Z}}\Big) \cong \Big(\frac{a_3,a_4}{\mathbb{Z}/p^k\mathbb{Z}}\Big)$$
and let $s\leq k$. If $a_i\equiv a'_i\pmod  {p^s}$ for every $1\leq i\leq 4$, then
$$ \Big(\frac{a'_1,a'_2}{\mathbb{Z}/p^s\mathbb{Z}}\Big) \cong \Big(\frac{a'_3,a'_4}{\mathbb{Z}/p^s\mathbb{Z}}\Big)$$
\end{Lemma}
\begin{proof} 
Let $f$ be an isomorphism between $\Big(\frac{a_1,a_2}{\mathbb{Z}/p^k\mathbb{Z}}\Big)$ and $\Big(\frac{a_3,a_4}{\mathbb{Z}/p^k\mathbb{Z}}\Big)$. If $A$ is the coordinate matrix of $f$ with respect to some standard bases, it is obvious that $A$ is regular over $\Z{p^k}$ and, consequently, also over $\Z{p^s}$.

Then, the linear map $g$ between $\Big(\frac{a'_1,a'_2}{\mathbb{Z}/p^s\mathbb{Z}}\Big)$ and $\Big(\frac{a'_3,a'_4}{\mathbb{Z}/p^s\mathbb{Z}}\Big)$ defined by the matrix $A$ with respect to some standard bases, induces an ring isomorphism because $a_i\equiv a'_i\pmod{p^s}$ for every $i$.
\end{proof}

It is also interesting, and often harder, to determine whether two quaternion rings are not isomorphic. The following results go in this direction.

\begin{Lemma}
\label{ceros} 
Let $p$ be a prime and consider integers $a$, $b$ and $c$ coprime to $p$. Also, let $0\leq s\leq r <k$. Then, the quaternion rings $R_1$, $R_2$ and $R_3$ defined by
$$R_1=\Big(\frac{a p^s,b p^r}{\mathbb{Z}/p^k\mathbb{Z}}\Big),\ R_2=\Big(\frac{c p^{s },0}{\mathbb{Z}/p^k\mathbb{Z}}\Big),\ R_3=\Big(\frac{0,0}{\mathbb{Z}/p^k\mathbb{Z}}\Big)$$
are pairwise non-isomorphic.
\end{Lemma}
\begin{proof}
For each $i\in\{1,2,3\}$ let us define the set $\mathbb{P}_i:=\{q \in R_i: {\rm tr}(q)=0\}$. 

Note that, if $p$ is odd, then $\mathbb{P}_i$ is precisely the set of pure quaternions and is hence preserved by isomorphisms. Now, for every element $q\in \mathbb{P}_3$ it holds that $q^2=0$, while $\mathbb{P}_1$ and $\mathbb{P}_2$ contain elements whose square is non-zero. This implies that $R_3$ is not isomorphic to $R_1$ or $R_2$. On the other hand, the set $\mathbb{P}_2\setminus p\mathbb{P}_2$ contains elements with zero square while this is not the case for $\mathbb{P}_1\setminus p\mathbb{P}_1$. This implies that $R_1$ and $R_2$ are not isomorphic.

Finally, if $p=2$, then $\mathbb{P}_i=\{\alpha 2^{k-1}+q_0: \alpha \textrm{ is odd and } q_0  \textrm{ is a pure quaternion}\}$ and we can reason in the same way.
\end{proof}

\begin{Lemma} \label{mismosordenes}  
Let $p$ be a prime and consider integers $a$, $b$, $c$ and $d$ coprime to $p$. Also, let $s_1\leq s_2\leq k$ and $s_3 \leq s_4\leq k$ and assume that either $s_1\neq s_3$ or $s_2\neq s_4$. Then
$$\Big(\frac{ap^{s_1} ,bp^{s_2}}{\Z{p^k}}\Big) \ncong \Big(\frac{cp^{s_3} ,dp^{s_4}}{\Z{p^k}}\Big) $$
\end{Lemma}
\begin{proof} 
Let us assume that both rings are isomorphic. Without loss of generality, we can also assume that $s_1 \leq s_3$. Five different situations arise:
\begin{itemize}
\item[i)]
If $s_1=s_3=s_2 < s_4$, then Lemma \ref{modulo} implies that
$$\Big(\frac{ap^{s_1} ,bp^{s_1}}{\Z{p^{s_4}}}\Big) \cong \Big(\frac{cp^{s_1}, 0}{\Z{p^{s_4}}}\Big),$$
which contradicts Lemma \ref{ceros}.
\item[ii)]
If $s_1=s_3<s_2<s_4$, then due to Lemma \ref{modulo} we have that
$$\Big(\frac{ap^{s_1} ,bp^{s_2}}{\Z{p^{s_4}}}\Big) \cong \Big(\frac{cp^{s_1}, 0}{\Z{p^{s_4}}}\Big),$$
which contradicts Lemma \ref{ceros}.
\item[iii)]
If $s_1=s_2<  s_3 $, by Lemma \ref{modulo} we have that
$$\Big(\frac{ap^{s_1} ,bp^{s_1}}{\Z{p^{s_3}}}\Big) \cong \Big(\frac{0, 0}{\Z{p^{s_3}}}\Big),$$
which contradicts Lemma \ref{ceros} again.
\item[iv)]
If $s_1<s_2\leq s_3$, Lemma \ref{modulo} implies that
$$\Big(\frac{ap^{s_1} ,0}{\Z{p^{s_2}}}\Big) \cong \Big(\frac{0, 0}{\Z{p^{s_2}}}\Big),$$
contradicting Lemma \ref{ceros}.
\item[v)]
If $s_1<s_3\leq s_2$, Lemma \ref{modulo} leads to
$$\Big(\frac{ap^{s_1} ,0}{\mathbb{Z}_{p^{s_3}}}\Big) \cong \Big(\frac{0, 0}{\Z{p^{s_3}}}\Big),$$
which is a contradiction due to Lemma \ref{ceros}.
\end{itemize}
Hence, in any case we reach a contradiction and the result follows.
\end{proof}

\section{Quaternions over $\Z{p^k}$ for an odd prime $p$}

This section is devoted to determine the number of non-isomorphic generalized quaternion rings over $\Z{p^k}$ for an odd prime $p$. Hence, throughout this section $p$ will be assumed to be an odd prime.

\begin{Lemma} \label{mismoresiduo} 
Let $s$ and $t$ be integers coprime to $p$ such that $st$ is a quadratic residue modulo $p$ and let $m$ be any integer. Then, for every $r\geq 0$,
$$R=\Big(\frac{t p^r,m}{\Z{p^k}}\Big)\cong\Big(\frac{s p^r,m}{\Z{p^k}}\Big)=S.$$
\end{Lemma}
\begin{proof} 
Since $\gcd(st,p)=1$, it follows that $st$ is also a quadratic residue modulo $p^k$ so let $x$ be an integer such that $x^2 \equiv t s^{-1} \pmod {p^k}$. Let us consider $\{1,i,j,k\}$ and $\{1,I,J,K\}$ standard bases of $R$ and $S$, respectively. Then, the linear map whose matrix with respect to these bases is
$$A=\begin{pmatrix} 1&0&0&0\\0&x&0&0\\0&0&1&0\\0&0&0&x\end{pmatrix}$$
induces an isomorphism because $(xI)^2=x^2I^2\equiv ts^{-1}sp^r\equiv tp^r\pmod{p^k}$ and $A$ is regular over $\Z{p^k}$.
\end{proof}

\begin{Lemma} \label{mismoresiduo2} 
Let $s$ be an integer such that $\gcd(p,s)=1$. Then, for every $r\geq 0$,
$$R=\Big(\frac{ p^r, p^r}{\Z{p^k}}\Big)\cong\Big(\frac{s p^r,s p^r}{\Z{p^k}}\Big)=S$$
\end{Lemma}
\begin{proof}   
Let $x,y \in \Z{p^k}^*$ such that $x^2+y^2 \equiv s^{-1} \pmod {p^k}$ (the exist due to \cite[Proposition 1]{nosotros}). Now let us consider $\{1,i,j,k\}$ and $\{1,I,J,K\}$ standard bases of $R$ and $S$, respectively. Then, the linear map whose matrix with respect to these bases is 
$$A=\begin{pmatrix} 1&0&0&0\\0&x&-y&0\\0&y&x&0\\0&0&0&s^{-1}\end{pmatrix}$$
induces an isomorphism because 
$$(x I+ y J)^2=(x^2+y^2)s p^r \equiv p^r \pmod {p^k},$$
$$(-y I+ x J)^2=(x^2+y^2)s p^r \equiv p^r \pmod {p^k},$$
$$(x I+ y J)(-y I+ x J) \equiv (x^2+y^2)K \equiv s^{-1}K \pmod {p^k}$$
and $A$ is regular over $\Z{p^k}$.
\end{proof}

\begin{Lemma} 
\label{casoescalar}  
Let $u$ be a quadratic nonresidue modulo $p$ with $p\nmid u$ and consider integers $a$ and $b$ coprime to $p$ and let $0\leq s$. Then,
\begin{itemize}
\item[i)] 
$$\Big(\frac{1, a p^s}{\Z{p^k}}\Big) \cong\Big(\frac{1,  p^s}{\Z{p^k}}\Big)\ \textrm{and}\ \Big(\frac{u,  p^s}{\Z{p^k}}\Big) \cong \Big(\frac{u, b p^s}{\Z{p^k}}\Big)$$
\item[ii)] The isomorphism
$$\Big(\frac{1,  p^s}{\Z{p^k}}\Big)\cong \Big(\frac{u,  p^s}{\Z{p^k}}\Big)$$
holds if and only if $s=0$.
\end{itemize}
\end{Lemma}
\begin{proof} 
\begin{itemize}
\item[i)]
To see that $\displaystyle R=\Big(\frac{1,  p^s}{\Z{p^k}}\Big) \cong \Big(\frac{1, a p^s}{\Z{p^k}}\Big)=S$, let us consider $\{1,i,j,k\}$ and $\{1,I,J,K\}$ standard bases of $R$ and $S$, respectively. It is obvious that there exist $x,y\in\Z{p^k}$ with $\gcd(p,y)=1$ such that $x^y-y^2\equiv a^{-1}\pmod{p^k}$. Then, the linear map whose matrix with respect to these bases is 
$$A=\begin{pmatrix} 1&0&0&0\\0&1&0&0\\0&0&x&y\\0&0&y&x\end{pmatrix}$$
induces an isomorphism because  
$$(xj+yk)^2=x^2 j^2+y^2k^2=x^2 ap^s-y^2ap^s=a p^s(x^2-y^2) \equiv  p^s \pmod {p^k}$$
and $A$ is regular over $\Z{p^k}$.

The remaining isomorphism can be proved in a similar way.
\item[ii)]
Assume that $s>0$. To see that $\displaystyle \Big(\frac{1,  p^s}{\Z{p^k}}\Big) \not \cong \Big(\frac{u,  p^s}{\Z{p^k}}\Big)$ it is enough to observe that $\displaystyle \Big(\frac{1,  p^s}{\Z{p^k}}\Big)$ does not contain any pure quaternion $q$ with $q^2=u$. In fact, if $\{1,i,j,k\}$ is a standard basis, $q=ai+bj+ck$ and $q^2=a^2+(b^2-c^2)p^s$. Hence, if $q^2\equiv u\pmod{p^k}$ if follows that $u$ is a quadratic residue modulo $p$, which is a contradiction.

On the other hand, if $s=0$, we know that $\Big(\frac{1,  1}{\Z{p^k}}\Big)\cong \Big(\frac{u,  1}{\Z{p^k}}\Big)$ using \cite[Theorem 4]{nosotros}.
\end{itemize}
\end{proof}

\begin{Lemma} 
\label{cerosres} 
Let $u$ be a quadratic nonresidue modulo $p$ with $p\nmid u$ and let $0<s<k$. Then,
\begin{itemize}
\item[i)]
$\displaystyle R_1=\Big(\frac{u p^s,  p^s}{\mathbb{Z}_{p^k}}\Big)\not\cong \Big(\frac{ p^s,  p^s}{\mathbb{Z}_{p^k}}\Big)=R_2.$
\item[ii)]
$\displaystyle S_1=\Big(\frac{u p^s,0}{\mathbb{Z}_{p^k}}\Big)\ncong \Big(\frac{ p^s,0}{\mathbb{Z}_{p^k}}\Big)=S_2.$
\end{itemize}
\end{Lemma}
\begin{proof}
\begin{itemize}
\item[i)]
Let us consider the following sets:
$$N_1:=\{q\in R_1 : \textrm{$q$ is a pure quaternion, $\textrm{n}(q)=0$, $pq\neq 0$}\},$$
$$N_2:=\{q\in R_2 : \textrm{$q$ is a pure quaternion, $\textrm{n}(q)=0$, $pq\neq 0$}\}.$$
In order to prove that $R_1\ncong R_2$ we will see that $\textrm{card}(N_1)\neq\textrm{card}(N_2)$.

To do so, let $\{1,i,j,k\}$ and $\{1,I,J,K\}$ be standard bases of $R_1$ and $R_2$, respectively. Then, if $q_1\in N_1$, it must be $q_1=x_1i+x_2j+x_3k$ with $x_1^2up^s+x_2^2p^s-x_3^sup^{2s}\equiv 0\pmod{p^k}$ and $p\nmid x_l$ for some $l\in\{1,2,3\}$. On the other hand, if $q_2\in N_2$, it must be $q_2=y_1I+y_2J+y_3K$ with $y_1^2p^s+y_2^2p^s-y_3^sp^{2s}\equiv 0\pmod{p^k}$ and $p\nmid y_l$ for some $l\in\{1,2,3\}$.

Thus, $\textrm{card}(R_1)$ is the number of non-zero solutions of the congruence 
\begin{equation}\label{e1}
x_1^2u+x_2^2-x_3^sup^{s}\equiv 0\pmod{p^{k-s}},\end{equation} 
while $\textrm{card}(R_2)$ is the number of non-zero solutions of the congruence 
\begin{equation}\label{e2}
y_1^2+y_2^2-y_3^sp^{s}\equiv 0\pmod{p^{k-s}}.\end{equation}

Now, reducing modulo $p$, we can see that:
\begin{itemize}
\item If $-1$ is a quadratic residue $\pmod{p}$ (i.e., if $p\equiv 1\pmod{4}$), then the congruence (\ref{e2}) has non-zero solutions while the congruence (\ref{e1}) has not.
\item If $-1$ is a quadratic nonresidue $\pmod{p}$ (i.e., if $p\equiv 3\pmod{4}$), then the congruence (\ref{e1}) has non-zero solutions while the congruence (\ref{e2}) has not.
\end{itemize}
In any case, it follows that $\textrm{card}(N_1)\neq\textrm{card}(N_2)$ as claimed.
\item[ii)]
For this case, it is enough to observe that $S_2$ does not contain pure quaternions $q$ such that $q^2=up^s$, while $S_1$ obviously does contain such type of elements. To do so, just note that the congruence $x^2p^s\equiv up^s\pmod{p^k}$ ha no solutions because $u$ is a quadratic nonresidue modulo $p$. 
\end{itemize}
\end{proof}

\begin{Lemma} \label{lema10}
Let $u$ be a quadratic nonresidue $\pmod p$ with $p\nmid u$ and let $0<s<r<k$. Then, the quaternion rings 
$\displaystyle R_1=\Big(\frac{ up^s,u p^r}{\Z{p^k}}\Big)$, $\displaystyle R_2=\Big(\frac{ p^s,u p^r}{\Z{p^k}}\Big)$, $\displaystyle R_3=\Big(\frac{ up^s,  p^r}{\Z{p^k}}\Big)$ and $\displaystyle R_4=\Big(\frac{p^s,  p^r}{\Z{p^k}}\Big)$ are pairwise non-isomorphic.
\end{Lemma}
\begin{proof} 
Let us see that $R_1\ncong R_2$, $R_1\ncong R_4$, $R_2\ncong R_3$ and $R_3\ncong R_4$. If they were isomorphic, the due to Lemma \ref{modulo} we would have (reducing modulo $p^r$) that $\displaystyle \Big(\frac{ up^s,  0}{\Z{p^r}}\Big)\cong \Big(\frac{ p^s,  0}{\Z{p^r}}\Big)$, which contradicts Lemma \ref{cerosres}.

Now, let us see that $R_1\ncong R_3$. Assume that $R_1\cong R_3$. Then, due to Proposition \ref{formamat}, we can consider $\{1,i,j,k\}$ and $\{1,I,J,K\}$ standard bases of $R_1$ and $R_3$, respectively such that the matrix of the isomorphism with respect to these bases is
$$\begin{pmatrix} 1&0&0&0\\0&1&\alpha_1&\alpha_2\\0&0&\beta_1&\beta_2\\0&0&\gamma_1&\gamma_2\end{pmatrix},$$
with $\alpha_1up^s=0$. 

In particular, $up^r=j^2=f(j^2)=f(j)^2=(\alpha_1I+\beta_1J+\gamma_1K)^2=\alpha_1^2up^s+\beta_1^2p^r-\gamma_1^2up^{r+s}=\beta_1^2p^r-\gamma_1^2u^2p^{r+s}$. In other words, $\beta_1^2p^r-\gamma_1^2up^{r+s}\equiv up^r\pmod{p^k}$ but this implies that $\beta_1^2-\gamma_1^2up^s\equiv u\pmod{p^{k-r}}$ and, consequently, that $\beta_1^2\equiv u\pmod{p}$ which is a contradiction because $u$ is a quadratic nonresidue.

The remaining case, namely $R_2\ncong R_4$ can be proved in the exact same way.
\end{proof}

\begin{cor} 
Let $u$ be a quadratic nonresidue modulo $p$ with $p\nmid u$. Consider integers $a$ and $b$ coprime to $p$ and let $0<r$. Then,
$$\Big(\frac{a,bp^r}{\Z{p^{k}}}\Big) \cong \begin{cases}\Big(\frac{u , p^r}{\Z{p^{k}}}\Big), & \textrm{if $a$ is a quadratic nonresidue modulo $p$}; \\ 
\Big(\frac{1  , p^r}{\Z{p^{k}}}\Big), & \textrm{if $a$ is a quadratic residue modulo $p$}.\end{cases}$$
\end{cor}
\begin{proof} 
If $a$ is a quadratic nonresidue:
$$\Big(\frac{a,bp^r}{\Z{p^{k}}}\Big) \substack{\cong\\\textrm{Lem. \ref{casoescalar}}} \Big(\frac{a,p^r}{\Z{p^{k}}}\Big) \substack{\cong\\\textrm{Lem. \ref{mismoresiduo}}} \Big(\frac{u,p^r}{\Z{p^{k}}}\Big).$$
Now, if $a$ is a quadratic residue:
$$\Big(\frac{a,bp^r}{\Z{p^{k}}}\Big) \substack{\cong\\\textrm{Lem. \ref{mismoresiduo}}} \Big(\frac{1,bp^r}{\Z{p^{k}}}\Big) \substack{\cong\\\textrm{Lem. \ref{casoescalar}}} \Big(\frac{1,p^r}{\Z{p^{k}}}\Big).$$
Finally, $\Big(\frac{u,p^r}{\Z{p^{k}}}\Big)$ and $\Big(\frac{1,p^r}{\Z{p^{k}}}\Big)$ are not isomorphic due to Lemma \ref{casoescalar}.
\end{proof}

\begin{cor} 
Let $u$ be a quadratic nonresidue modulo $p$ with $p\nmid u$. Consider integers $a$ and $b$ coprime to $p$ and let $0<r$. Then,
$$\Big(\frac{ap^r ,bp^r}{\Z{p^{k}}}\Big) \cong \begin{cases}\Big(\frac{up^r , p^r}{\Z{p^{k}}}\Big), & \textrm{if $ab$ is a quadratic nonresidue modulo $p$}; \\ \Big(\frac{p^r,p^r}{\Z{p^{k}}}\Big), & \textrm{if $ab$ is a quadratic residue modulo $p$}.\end{cases}$$
\end{cor}
\begin{proof} 
If $ab$ is a quadratic nonresidue, only one among $a$ and $b$ is a quadratic residue. We can assume without loss of generality that $a$ is a quadratic residue and that $b$ is a quadratic nonresidue (so $ub$ is a quadratic residue) and then:
$$\Big(\frac{ap^r,bp^r}{\Z{p^{k}}}\Big) \substack{\cong\\\textrm{Lem. \ref{mismoresiduo}}} \Big(\frac{ap^r,up^r}{\Z{p^{k}}}\Big) \substack{\cong\\\textrm{Lem. \ref{mismoresiduo}}} \Big(\frac{p^r,up^r}{\Z{p^{k}}}\Big).$$
Now, if $ab$ is a quadratic residue:
$$\Big(\frac{ap^r,bp^r}{\Z{p^{k}}}\Big) \substack{\cong\\\textrm{Lem. \ref{mismoresiduo}}} \Big(\frac{bp^r,bp^r}{\Z{p^{k}}}\Big) \substack{\cong\\\textrm{Lem. \ref{mismoresiduo2}}} \Big(\frac{p^r,p^r}{\Z{p^{k}}}\Big).$$
Finally, $\Big(\frac{p^r,p^r}{\Z{p^{k}}}\Big)$ and $\Big(\frac{up^r,p^r}{\Z{p^{k}}}\Big)$ are not isomorphic due to Lemma \ref{cerosres}.
\end{proof}

\begin{cor} 
Let $u$ be a quadratic nonresidue modulo $p$ with $p\nmid u$. Consider integers $a$ and $b$ coprime to $p$ and let $0<s<r$. Then,
$$\Big(\frac{ap^s,bp^r}{\Z{p^{k}}}\Big) \cong \begin{cases}\Big(\frac{up^s,p^r}{\Z{p^{k}}}\Big), & \textrm{if only $b$ is a quadratic residue modulo $p$}; \\ \Big(\frac{p^s,u p^r}{\Z{p^{k}}}\Big),& \textrm{if only $a$ is a quadratic residue $\pmod p$}.  \\ \Big(\frac{p^s,p^r}{\Z{p^{k}}}\Big), & \textrm{if both $a$ and $b$ are quadratic residues modulo $p$}. \\ \Big(\frac{up^s ,u p^r}{\Z{p^{k}}}\Big), & \textrm{if both $a$ and $b$ are quadratic nonresidues modulo $p$}. \end{cases} $$
\end{cor}
\begin{proof} 
Like in the previous results, it is enough to apply Lemma \ref{mismoresiduo} repeatedly. The four different cases that arise are non-isomorphic due to Lemma \ref{lema10}.
\end{proof}

Now, we can prove the main result of this section.

\begin{teor}\label{TEORPP}
Let $p$ be an odd prime and let $k$ be a positive integer. Then, there exist exactly $2k^2+2$ non-isomorphic generalized quaternion rings over $\Z{p^k}$.
\end{teor}
\begin{proof} 
Taking into account the previous results, any generalized quaternion ring over $\Z{p^k}$ is isomorphic to one of the following:
$$\Big(\frac{up^s,up^r}{\Z{p^k}}\Big), \Big(\frac{p^s,up^r}{\Z{p^k}}\Big), \Big(\frac{up^s, p^r}{\Z{p^k}}\Big), \Big(\frac{p^s,p^r}{\Z{p^k}}\Big),$$
where $u$ is a quadratic nonresidue $\pmod p$ with $p\nmid u$ and $0\leq s\leq r\leq k$. 
\begin{itemize}
\item If $0=s=r$, due to Lemmata \ref{orden}, \ref{mismoresiduo2} and \ref{casoescalar}, there is only one ring to consider, namely $\Big(\frac{1,1}{\Z{p^k}}\Big)$.
\item If $0=s<r<k$, we must consider the rings
$$\Big(\frac{u,up^r}{\Z{p^k}}\Big), \Big(\frac{1,up^r}{\Z{p^k}}\Big), \Big(\frac{u, p^r}{\Z{p^k}}\Big), \Big(\frac{1,p^r}{\Z{p^k}}\Big).$$
Due to Lemma \ref{casoescalar} we know that $\Big(\frac{u,up^r}{\Z{p^k}}\Big)\cong\Big(\frac{u, p^r}{\Z{p^k}}\Big)$, $\Big(\frac{1,up^r}{\Z{p^k}}\Big)\cong\Big(\frac{1,p^r}{\Z{p^k}}\Big)$ and $\Big(\frac{u, p^r}{\Z{p^k}}\Big)\ncong \Big(\frac{1,p^r}{\Z{p^k}}\Big)$. Hence, in this case we have 2 non-isomorphic generalized quaternion rings for each $1\leq r\leq k-1$. A total of $2(k-1)$.
\item If $0=s$ and $k=r$ we must only consider the rings 
$$\Big(\frac{u,0}{\Z{p^k}}\Big), \Big(\frac{1,0}{\Z{p^k}}\Big)$$
which are non-isomorphic due to Lemma \ref{casoescalar}. Thus, in this case we have 2 non-isomorphic generalized quaternion rings.
\item If $0<s=r<k$, we must consider the rings
$$\Big(\frac{up^r,up^r}{\Z{p^k}}\Big), \Big(\frac{p^r,up^r}{\Z{p^k}}\Big), \Big(\frac{up^r, p^r}{\Z{p^k}}\Big), \Big(\frac{p^r,p^r}{\Z{p^k}}\Big).$$
Using Lemma \ref{orden}, Lemma \ref{mismoresiduo2} and Lemma \ref{cerosres} we know that $\Big(\frac{up^r,up^r}{\Z{p^k}}\Big)\cong \Big(\frac{p^r,p^r}{\Z{p^k}}\Big)\ncong \Big(\frac{up^r, p^r}{\Z{p^k}}\Big)\cong \Big(\frac{p^r,up^r}{\Z{p^k}}\Big)$. Hence, in this case we have 2 non-isomorphic generalized quaternion rings for each $1\leq r\leq k-1$ for a total of $2(k-1)$.
\item If $0<s<r<k$, Lemma \ref{lema10} implies that the four rings are non-isomorphic. Hence, in this case we have 2 non-isomorphic generalized quaternion rings for each $1\leq s\leq k-2$ and each $s+1\leq k\leq k-1$. A total of $2(k-2)(k-1)$.
\item If $0<s<r=k$, we must only consider the rings 
$$\Big(\frac{up^s,0}{\Z{p^k}}\Big), \Big(\frac{p^s,0}{\Z{p^k}}\Big)$$
which are non-isomorphic due to Lemma \ref{cerosres}. Thus, in this case we have 2 non-isomorphic generalized quaternion rings for each $1\leq s\leq k-1$. A total of $2(k-1)$.
\item If $s=r=k$ there is only one ring to consider, namely $\Big(\frac{0,0}{\Z{p^k}}\Big)$.
\end{itemize}
Finally, taking into consideration all the previous information, we conclude that there exist
$$1+2(k-1)+2+2(k-1)+2(k-2)(k-1)+2(k-1)+1=2k^2+2$$
non-isomorphic generalized quaternion rings over $\Z{p^k}$.
\end{proof}

\begin{rem}
The sequence $a_k=2k^2+2$ is sequence A005893 in the OEIS.
\end{rem}

\section{Quaternions over $\Z{n}$ for an odd $n$}

Note that if $n=p_1^{r_1}\ldots p_k^{r_k}$ is the prime factorization of $n$, then by the Chinese Remainder Theorem we have that
\begin{equation}\label{l3}\Z{n}\cong \Z{p_1^{r_1}}\oplus\ldots\oplus\Z{p_k^{r_k}}.\end{equation}
Decomposition (\ref{l3}) induces a natural isomorphism
\begin{equation} \label{F1}\Big(\frac{a,b}{\Z{n}} \Big) \cong \Big(\frac{a,b}{\Z{p_1^{r_1}}} \Big) \oplus\ldots\oplus \Big(\frac{a,b}{\Z{p_k^{r_k}}} \Big) .\end{equation}
 
Consequently, if we denote by $\omega(n)$ the number of different primes dividing $n$ and by $\nu_p(n)$ the $p$-adic order of $n$ we obtain the following corollary to Theorem \ref{TEORPP}.

\begin{cor}
Let $n$ be an odd integer. Then, the number of non-isomorphic generalized quaternion rings over $\Z{n}$ is
$$2^{\omega(n)}\prod_{p\mid n} (\nu_p(n)^2+1).$$
\end{cor}

\end{document}